\documentclass[11pt]{amsart}
\usepackage{amscd}
\usepackage{color}
\usepackage{amsfonts}
\usepackage[all]{xy}
\usepackage{amsmath,amssymb,amsthm,latexsym,mathdots}
\usepackage{amscd}

\usepackage{bm}
\usepackage{setspace}
\usepackage{caption}
\newtheorem{theorem}{Theorem}[section]

\newtheorem{definition}[theorem]{Definition}
\newtheorem{corollary}[theorem]{Corollary}
\newtheorem{lemma}[theorem]{Lemma}
\numberwithin{equation}{section}
\theoremstyle{remark}
\newtheorem{remark}[theorem]{Remark}
\newtheorem{example}[theorem]{\bf Example}
\newcommand{\R}{\mathbb{R}}
\newcommand{\C}{\mathbb{C}}
\newcommand{\D}{\mathbb{D}}

\newcommand{\FC}{\mathcal{C}}
 \newcommand{\dd}{\mathrm{d}}

\begin{document}

\title[Coarse classification of Willmore two-spheres]{\bf{Willmore surfaces in spheres via loop groups II: a coarse classification of Willmore two-spheres by potentials}}
\author{Peng Wang}
\date{}
\maketitle

\begin{center}
{\bf Abstract}

\end{center}

\ \\
Applying the DPW version of  the theory developed by Burstall and Guest for harmonic maps of finite uniton type, we derive a coarse classification of Willmore two-spheres in $S^{n+2}$ in terms of the normalized potential of their (harmonic) conformal Gauss maps.
Moreover, for the case of $S^6$, some geometric properties of the corresponding Willmore two-spheres are discussed. The classical classification of Willmore two-spheres in $S^4$ are also derived as a corollary. \\

{\bf Keywords:}  harmonic maps of finite uniton; Willmore two-spheres; normalized potential; S-Willmore surfaces; totally isotropic surfaces.\\

MSC(2010): 58E20; 53C43;  	53A30;  	53C35\\

\section{Introduction}

Willmore surfaces can be looked at as surfaces sharing a best conformal placement in $S^{n+2}$,
since they are critical surfaces of the conformally invariant Willmore functional $$\int_M(H^2-K+1)\mathrm{d}M.$$ It is well known that minimal surfaces in the three space forms $\mathbb{R}^{n+2}$, $S^{n+2}$ and $\mathbb{H}^{n+2}$ are all Willmore surfaces (see e.g. \cite{Ejiri1988}, \cite{Helein}), providing a large class of examples and indicating the high variety and complexity of Willmore surfaces.

The classical theorem, due to Bryant, states that every Willmore two-sphere in $S^3$
is conformally equivalent to  a minimal surface in $\mathbb{R}^3$ with embedded ends.
This work has led to a systematical investigation of Willmore two-spheres as well as Willmore surfaces in $S^{n+2}$. In this context, many interesting geometric objects with ``nice behavior"  has been introduced, including minimal surfaces with embedded planer ends (\cite{Bryant1984}, \cite{Bryant1988}) as well as holomorphic curves or anti-holomorphic curves in the twistor bundle of $S^4$ (\cite{Ejiri1988}, \cite{Mon}, \cite{Mus1}).

Contributions of Bryant's work also include proof of the harmonicity of the {\em conformal Gauss map} of a Willmore surface, the introduce of dual (Willmore) surfaces and holomorphic forms related to Willmore surfaces \cite{Bryant1984}. To study Willmore surfaces in higher dimensional spheres $S^{n+2}$, Ejiri generalized Bryant's definition of conformal Gauss map and showed the equivalence of surfaces being Willmore and having a harmonic conformal Gauss map. However the duality properties of Willmore surfaces no longer hold for general Willmore surfaces in $S^{n+2}$ when $n>1$ (for a generalization of duality we refer to \cite{BFLPP}, \cite{Ma} and \cite{Ma2006}). To use the duality properties, Ejiri introduced the notion of S-Willmore surfaces and proved that a Willmore surface has a dual (Willmore) surface if and only if it is S-Willmore. Moreover, Ejiri provided a classification of S-Willmore (Willmore) two-spheres in $S^{n+2}$ by the construction of holomorphic forms for these surfaces \cite{Ejiri1988}. Especially, for Willmore two-spheres in $S^4$, the existence of a certain holomorphic 4-form shows that they are S-Willmore automatically (\cite{Ejiri1988}, \cite{Mus1}, \cite{Mon}, \cite{BFLPP}, \cite{Ma}).

Since Ejiri's work in 1988, it has been an open question that whether are there Willmore two-spheres in $S^{n+2}$ which are not S-Willmore. If such Willmore two-spheres  exist,  then one needs to show how to construct and to characterize them.\\

As stated in the title, as an application of \cite{DoWa1} and \cite{DoWa2}, this paper aims to provide a coarse classification of Willmore two-spheres by the
loop group method. Moreover, in \cite{Wang-3}, the construction of new Willmore spheres will be presented. Moreover, a concrete, new, isotropic Willmore two-sphere in $S^6$, without any dual surface, is constructed as an illustration of the theory \cite{Wang-3}.

The techniques we apply here are based on the classification theory of harmonic maps of finite unition type \cite{BuGu}, \cite{Gu2002},
 and the loop group method for the construction of (conformally) harmonic maps \cite{DPW}, \cite{DoWa1}.
 The vital idea is that, since there has been a complete description of harmonic two-spheres in compact symmetric spaces (\cite{Uh}, \cite{BuGu}, \cite{Gu2002}), one should be able to modify this method to give a classification as well as examples of the conformal Gauss map of Willmore two-spheres.
  Along this way one will obtain the classification  as well as new examples of Willmore two-spheres.
  There are several difficulties one has to overcome before one is able to implement this idea. The first one is that the target manifold of the conformal Gauss map is a non-compact symmetric space.  The second one is how to read off the properties of Willmore surfaces from the harmonic conformal Gauss maps. The third one is how to modify the theory of Burstall and Guest to harmonic maps into non-compact Lie groups in the framework of DPW. And the last one is how to do the explicit Iwasawa decompositions required by the loop group method. The first two problems has been solved in \cite{DoWa1} and the third problem has been solved in \cite{DoWa2}.  This leads to the possibility to solve the problem considered in this paper under the framework of \cite{DoWa1},
by using the method introduced in \cite{DPW} and \cite{BuGu} (see \cite{DoWa2}).

According to the theory due to Uhlenbeck \cite{Uh}, Segal \cite{Segal}, Burstall and Guest \cite{BuGu,Gu2002},
any harmonic map from $S^2$ to $G/K$ is of finite uniton type, where $G/K$ a compact or non-compact inner symmetric space (For the case of non-compact $G/K$, see \cite{DoWa1,DoWa2}). Moreover, for any harmonic map of finite uniton type, there exists a normalized potential which takes values in some nilpotent Lie subalgebra. And the corresponding
normalized framing takes values in the corresponding Lie subgroup, which can be determined by the recipe of \cite{BuGu}.
 Therefore, a classification of Willmore two spheres in $S^{n+2}$ is equivalent to classify the corresponding nilpotent Lie subalgebras related to the Grassmannian
 where the conformal Gauss map takes value in.
 Since in the Willmore case, $G/K=SO^+(1,n+3)/(SO^+(1,3)\times SO(n))$, we will classify by the procedure stated in \cite{BuGu} all nilpotent Lie subalgebras which are associated with this symmetric space.  See Theorem \ref{thm-potential} for details.  Especially, as an application, we can obtain a classification of Willmore 2-spheres into $S^6$ in Theorem \ref{thm-potential-s6}.
 To read off concrete inform about the harmonic maps of Willmore surfaces as well as the explicit expressions of these Willmore two-spheres,
 one needs to carry out Iwasawa decompositions in a concrete fashion. Since these involve many tedious computations, we leave them  to  \cite{Wang-2,Wang-3}.

We also notice that there have been recently several publications on harmonic maps of finite uniton type into compact groups like $U(n)$ and $G_2$ (see \cite{Co-Pa}, \cite{FSW}, and \cite{S-W} and reference therein). Here we follow the spirits of Burstall-Guest and DPW, which is different from the work just mentioned. Another important different point is that the group we consider is non-compact while the groups appeared in \cite{Co-Pa, FSW, S-W} are compact. \\

This paper is organized as follows. We first review the most necessary background and formulations briefly in Section 2.
Then Section 3 contains the main result and its applications,
including the classification of nilpotent normalized potentials,
applications to Willmore two-spheres in $S^6$,
and examples of Willmore surfaces of finite uniton type.
 Section 4 ends the paper with a technical proof of the main result.

\section{Review of Willmore surfaces and loop group methods}

In this section, for the readers' convenience, we will collect the necessary results Willmore surfaces and the loop group methods used in this paper. We refer to \cite{DoWa1} and \cite{DoWa2} for more details. For the loop group theory, see also \cite{BuGu, DPW, Gu2002, Uh}.

\subsection{Willmore surfaces and strongly conformally harmonic maps}

we first recall some basic notation and results of \cite{DoWa1}.
Note the treatment in \cite{DoWa1} is different from the ones applied in \cite{Helein,Helein2} and \cite{Xia-Shen}, where  a different kind of harmonic maps (first introduced by H\'{e}lein \cite{Helein}) are mainly used.

Let $\mathbb{R}^{n+4}_1$ be the Lorentz-Minkowski space equipped with the Lorentzian metric
$\langle x,y\rangle=-x_{0}y_0+\sum_{j=1}^{n+3}x_jy_j=x^t I_{1,n+3} y,\ \  I_{1,n+3}=diag\{-1,1,\cdots,1\}, \ \
  x,y\in\R^{n+4}.$
Let $\mathcal{C}^{n+3}_+= \lbrace x \in \mathbb{R}^{n+4}_{1} |\langle x,x\rangle=0 , x_0 >0 \rbrace $
denote the forward light cone of $\mathbb{R}^{n+4}_{1}$ and $Q^{n+2}=\mathcal{C}^{n+3}_+/ \R^+=S^{n+2}$ denote the  projective light cone.
For a  conformal immersion $y:M\rightarrow S^{n+2}$ from a Riemann surface $M$, one has a canonical lift $Y=e^{-\omega}(1,y)$ into $\mathcal{C}^{n+3}$ with respect to a local complex coordinate $z$ of $M$, where $e^{2\omega}=2\langle y_z,y_{\bar{z}}\rangle$.
It is easily verified that there is  a global bundle decomposition
\begin{equation}
M\times \mathbb{R}^{n+4}_{1}=V\oplus V^{\perp}, \ \hbox{ with }\ V={\rm Span}\{Y,{\rm Re}Y_{z},{\rm Im}Y_{z},Y_{z\bar{z}}\},
\end{equation}
where $V^{\perp}$ denotes the orthogonal complement of $V$.
The complexifications of $V$ and $V^{\perp}$ are denoted by $V_{\mathbb{C}}$ and
$V^{\perp}_{\mathbb{C}}$ respectively.
Choose the frame $\{Y,Y_{z},Y_{\bar{z}},N\}$ of
$V_{\mathbb{C}}$, where $N$ is the uniquely determined section of $V$ over $M$ satisfying
$ \langle N,Y_{z}\rangle=\langle N,Y_{\bar{z}}\rangle=\langle
N,N\rangle=0,\langle N,Y\rangle=-1.$ Let $D$
denote the normal connection on $V_{\mathbb{C}}^{\perp}$. For any section $\psi\in
\Gamma(V_{\mathbb{C}}^{\perp})$ of the normal bundle
and a canonical lift $Y$ w.r.t $z$, we obtain
the structure equations:
\begin{equation}\label{eq-moving}
\left\{\begin {array}{lllll}
Y_{zz}=-\frac{s}{2}Y+\kappa,\\
Y_{z\bar{z}}=-\langle \kappa,\bar\kappa\rangle Y+\frac{1}{2}N,\\
N_{z}=-2\langle \kappa,\bar\kappa\rangle Y_{z}-sY_{\bar{z}}+2D_{\bar{z}}\kappa,\\
\psi_{z}=D_{z}\psi+2\langle \psi,D_{\bar{z}}\kappa\rangle Y-2\langle
\psi,\kappa\rangle Y_{\bar{z}},
\end {array}\right.
\end{equation}
Here $\kappa$ is  \emph{the conformal Hopf differential} of $y$,
and $s$ is \emph{the Schwarzian} of $y$ \cite{BPP}. For the structure equations \eqref{eq-moving} the integrability conditions are
the conformal Gauss, Codazzi and Ricci equations respectively:
\begin{equation}\label{eq-integ}
\left\{\begin {array}{lllll} \frac{1}{2}s_{\bar{z}}=3\langle
\kappa,D_z\bar\kappa\rangle +\langle D_z\kappa,\bar\kappa\rangle,\\
{\rm Im}(D_{\bar{z}}D_{\bar{z}}\kappa+\frac{\bar{s}}{2}\kappa)=0,\\
R^{D}_{\bar{z}z}\psi=D_{\bar{z}}D_{z}\psi-D_{z}D_{\bar{z}}\psi =
2\langle \psi,\kappa\rangle\bar{\kappa}- 2\langle
\psi,\bar{\kappa}\rangle\kappa.
\end {array}\right.
\end{equation}

\emph{The Willmore functional} of $y$ is
defined as  :
\begin{equation}\label{eq-W-energy}
W(y):=2i\int_{M}\langle \kappa,\bar{\kappa}\rangle \mathrm{d}z\wedge
\mathrm{d}\bar{z}.
\end{equation}
An immersed surface $y:M\rightarrow S^{n+2}$ is called a
\emph{Willmore surface}, if it is a critical point of the Willmore
functional with respect to any variation (with compact support) of the map $y:M\rightarrow
 S^{n+2}$.
  Set
\begin{equation}\label{eq-gauss}
Gr:=Y\wedge Y_{u}\wedge Y_{v}\wedge N=-2i\cdot Y\wedge Y_{z}\wedge
Y_{\bar{z}} \wedge N.
\end{equation}
It is easy to verify that $Gr$ is well defined. We call
\[Gr:M\rightarrow
Gr_{1,3}(\mathbb{R}^{n+4}_{1}) = SO^+(1,n+3)/SO^+(1,3)\times SO(n)\] \emph{the (oriented) conformal Gauss map} of
$y$ (See also  \cite{Bryant1984, BPP, Ejiri1988,Ma}). Here $SO^+(1,n+3)$ is the oriention-preserving and timelike direction-preserving isometry group of $\mathbb{R}^{n+4}_1$:
\begin{equation}\begin{split}
SO^+(1,n+3)=&\{T\in Mat(n+4,\R)\ | \\
  &~~~~\hspace{2mm} \langle Tx,Ty\rangle=\langle x,y\rangle, \forall x,y\in\mathbb{R}^{n+4}_{1}, \det T=1, T\mathcal{C}^{n+3}_+\subset\mathcal{C}^{n+3}_+\}.\\
  \end{split}
\end{equation}

  It is well-known that \cite{Bryant1984,BPP,Ejiri1988} $y$ is a Willmore surface if and only if
   its conformal Gauss map $Gr$   is a (conformally) harmonic map into
$G_{3,1}(\mathbb{R}^{n+3}_{1})$. \\

By \cite{Ejiri1988},  a Willmore immersion $y:M^2\rightarrow S^n$ is called an \emph{S--Willmore} surface if its conformal Hopf differential satisfies
$D_{\bar{z}}\kappa || \kappa,$
i.e. there exists some function $\mu$ on $M$ such that $D_{\bar{z}}\kappa+\mu\kappa=0$.
It is known that  $($\cite{Ejiri1988},  \cite{Ma}$)$ a Willmore surface $y$ is S--Willmore  if and only if it has a dual (Willmore) surface.

\begin{remark} Although there are restrictions for Willmore surfaces to have a dual surface, it is known that there exist special transforms for Willmore surfaces, called ``adjoint transforms", see \cite{Ma2006}, \cite{Ma} for details. These transformations have a close relation  with Helein's work \cite{Helein} (See  \cite{DoWa1} for some details).
\end{remark}

There exist Willmore surfaces which fail to be immersions at some points or even curves. To include surfaces of this type, we introduce the notion of {\em Willmore maps} and {\em strong  Willmore maps} as generalizations.
\begin{definition}
A smooth map $y$ from a Riemann surface $M$ to $S^{n+2}$ is called a Willmore map if it is a conformal Willmore immersion on an open dense subset $\hat{M}$ of $M$. The points in $M_0=M\backslash \hat{M}$ are called branch points of $y$.

Moreover, $y$ is called a strong Willmore map if the conformal Gauss map $Gr: \hat{M}\rightarrow SO^+(1,n+3)/SO^+(1,3)\times SO(n)$ of $y$ can be extended smoothly to $M$.
\end{definition}

As a consequence, one can go from a strong Willmore map to a harmonic conformal map. To describe such harmonic maps and to characterize those harmonic maps related to Willmore surfaces, we first recall from \cite{DoWa1} that for any strong Willmore map $y$, locally on $U\subset M$ we can choose a frame
\begin{equation}\label{F}
F:=\left(\frac{1}{\sqrt{2}}(Y+N),\frac{1}{\sqrt{2}}(-Y+N),e_1,e_2,\psi_1,\cdots,\psi_n\right)\in SO^+(1,n+3)
\end{equation}
with the Maurer-Cartan form
 \[\alpha=F^{-1}\mathrm{d}F=\left(
                   \begin{array}{cc}
                     A_1 & B_1 \\
                   -B_1^tI_{1,3} & A_2 \\
                   \end{array}
                 \right)\mathrm{d}z+\left(
                   \begin{array}{cc}
                     \bar{A}_1 & \bar{B}_1 \\
                    -\bar{B}_1^tI_{1,3}& \bar{A}_2 \\
                   \end{array}
                 \right)\mathrm{d}\bar{z},\]
where
\begin{equation} B_1=\left(
      \begin{array}{ccc}
         \sqrt{2} \beta_1 & \cdots & \sqrt{2}\beta_n \\
         -\sqrt{2} \beta_1 & \cdots & -\sqrt{2}\beta_n \\
        -k_1 & \cdots & -k_n \\
        -ik_1 & \cdots & -ik_n \\
      \end{array}
    \right). \end{equation} Here $\{\psi_j\}$ is an orthonormal basis of $V^{\perp}$ on $U$ and
  $\kappa=\sum k_j\psi_j$, $D_{\bar{z}}\kappa=\sum_j\beta_j\psi_j$. And the entries of $A_1$ and $A_2$ are determined by $s$, $\kappa$ and the normal connection (See \cite{DoWa1}).

  It is easy to see that $B_1$ has a special form. Moreover, a direct computation shows $$B_1^t I_{1,3}B_1=0.$$
In \cite{DoWa1}  this property of $B_1$  plays an important role in the characterization of harmonic maps related to Willmore surfaces . We also point out that  $y$ is S--Willmore if and only if $rank(B_1)=1$ on an open dense subset of $M$.

    Conversely, assume we have the frame $F=(\phi_1,\cdots,\phi_4,\psi_1,\cdots,\psi_{n+4}):U\rightarrow SO^+(1,n+3)$ of some immersion on $U\subset M$,
such that the Maurer-Cartan form $\alpha=F^{-1}\mathrm{d}F$  of $F$ is of the above form,
then
\begin{equation}y=\pi_0(F)=:\left[ (\phi_1-\phi_2)\right]
\end{equation}
is a conformal immersion from $U$ into $Q^{n+2}\cong S^{n+2}$ (with canonical lift $\frac{1}{\sqrt{2}}(\phi_1-\phi_2)$).

Next, the symmetric space $SO^+(1,n+3)/SO^+(1,3)\times SO(n)$ is defined by the involution
\begin{equation} \sigma:SO^+(1,n+3)\rightarrow SO^+(1,n+3), \ \ \ ~~\sigma (A)=DAD^{-1},
\end{equation}
with $D=\hbox{diag}\{-I_4,I_n\}$. Then
\[\begin{split}&\mathfrak{k}=\left\{\left.\left(\begin{array}{cc}
                     A_1 &0 \\
                    0 & A_2 \\
                   \end{array}
                 \right)\right|I_{1,3}A_1+A_1^tI_{1,3}=0, A_2^t+A_2=0\right\},\\ &\mathfrak{p}=\left\{\left.\left(\begin{array}{cc}
                     0 &B_1 \\
                    -B_1^tI_{1,3} & 0 \\
                   \end{array}
                 \right)\right|B_1\in Mat(4\times n, \mathbb{R})\right\}.
                 \end{split}\]
Let $\mathcal{F}: M\rightarrow SO^+(1,n+3)/SO^+(1,3)\times SO(n)$ be a harmonic map. Then it has a local lift $F:U\rightarrow SO^+(1,n+3)$, and the Maurer-Cartan form $\alpha=F^{-1}\mathrm{d}F$ of $F$ is of the form $$\alpha=\left(
                   \begin{array}{cc}
                     A_1 & B_1 \\
                     B_2 & A_2 \\
                   \end{array}
                 \right)\mathrm{d}z+\left(
                   \begin{array}{cc}
                     \bar{A}_1 & \bar{B}_1 \\
                     \bar{B}_2 & \bar{A}_2 \\
                   \end{array}
                 \right)\mathrm{d}\bar{z}.$$

\begin{definition}\cite{DoWa1} $\mathcal{F}$ is called  a {\bf strongly conformally harmonic map} if for any point $p\in M$, there exists a neighborhood $U_p$ of $p$ and a frame $F$ (with Maurer-Cartan forms $\alpha$) of $\mathcal{F}$ on $U_p$ satisfying
\begin{equation}\label{eq-Willmore harmonic}B_1^t I_{1,3} B_1 = 0.\end{equation}
 \end{definition}

\begin{remark}\
\begin{enumerate}  \item We see that the conformal Gauss map of any Willmore surface is a strongly conformally harmonic map. Conversely,
Theorem 3.10 of \cite{DoWa1} shows that from a strongly conformally harmonic map, one either obtains a (branched) Willmore surface, or a constant map. In \cite{Wang-2}, we have classified the potentials of all strongly conformally harmonic maps yielding a constant map. Moreover, we prove that Willmore surfaces having such conformal Gauss maps must be conformally equivalent to some minimal surfaces in $\R^m$. Together with this result,  we obtain a way to characterize all Willmore surfaces globally by strongly conformally harmonic maps.

\item The conformal Gauss map defined in \eqref{eq-gauss} maps every point to an oriented 4--dimensional Lorentzian subspace. So in general one need to be careful with the orientation \cite{DoWa1}. Note that the orientation can be reversed by use of conjugation by some element in $O(1,3)\times O(n)$. In our cases, since we are always dealing with the harmonic maps on the Lie algebra level, and the conjugation will not change too much, so in many cases, we will ignore this difference and just state our results up to a conjugation of some element in $O(1,3)\times O(n)$.
\end{enumerate}\end{remark}
\subsection{The DPW method and Burstall-Guest theroy for harmonic maps of finite uniton type}

It is well known that harmonic maps into symmetric spaces inherit some $S^1-$parameter (loop) symmetries , which provides a way to describe them using  algebraic tools \cite{Uh}, \cite{DPW}. To be concrete, the basic idea is to describe  harmonic maps into a symmetric space $G/K$ by some special meromorphic 1-forms, via the Birkhoff decomposition and the Iwasawa decomposition of the loop group associated with $G$.
Moreover, for harmonic maps of finite uniton type, the theory of Burstall Guest can be summarised briefly that the corresponding special meromorphic 1-forms take values in some nilpotent Lie sub-algebras related to the symmetric space. Therefore, the classification of  harmonic maps of finite uniton type reduces to the classification of  the special nilpotent Lie sub-algebras. Together with the global relations between Willmore surfaces and harmonic maps, we are able to derive a classification of Willmore two-spheres in this way.

  \subsubsection{The DPW construction for harmonic maps}
Let $G$ be a connected, real, semi-simple non-compact matrix Lie group. Let $G/K$ be the inner symmetric space defined by the involution $\sigma: G\rightarrow G$, with $Fix^{\sigma} G\supseteq K \supseteq (Fix^{\sigma} G )^\circ$. Here $H^\circ$ denotes the identity component of the group $H$.
Let $\mathfrak{g}$, $\mathfrak{k}$ be the Lie algebras of $G$ and $K$. The Cartan decomposition shows that $$\hbox{$\mathfrak{g}=\mathfrak{k}\oplus\mathfrak{p}$  with
$ [\mathfrak{k},\mathfrak{k}]\subset\mathfrak{k},
~~~ [\mathfrak{k},\mathfrak{p}]\subset\mathfrak{p}, ~~~
[\mathfrak{p},\mathfrak{p}]\subset\mathfrak{k}.$}$$
Let $\pi:G\rightarrow G/K$ denote the projection of $G$ onto $G/K$.

 Let $\mathcal{F}:M\rightarrow G/K$ be a conformal harmonic map.
Let $U\subset M$ be an open contractible subset.
Then there exists a frame $F: U\rightarrow G$ such that $\mathcal{F}=\pi\circ F$. One has the Maurer-Cartan form $\alpha=F^{-1} \dd F$ and the Maurer-Cartan equation $\dd \alpha+\frac{1}{2}[\alpha\wedge\alpha]=0.$
Moreover, decomposing $\alpha$ with respect to $\mathfrak{g}=\mathfrak{k}\oplus\mathfrak{p}$ and  the complexification $T^*M^{\mathbb{C}}=T^*M'\oplus T^*M''$, we obtain
\[\alpha=\alpha_{\mathfrak{p}}'+\alpha_{ \mathfrak{k}  } +\alpha_{\mathfrak{p}}'', \hbox{ with }
\alpha_{\mathfrak{k  }}\in \Gamma(\mathfrak{k}\otimes T^*M),\
\alpha_{ \mathfrak{p }}'\in \Gamma(\mathfrak{p}^{\C}\otimes T^*M'),\
\alpha_{ \mathfrak{p }}''\in \Gamma(\mathfrak{p}^{\C}\otimes T^*M'').\] Set
$
\alpha_{\lambda}=
\lambda^{-1}\alpha_{\mathfrak{p}}'+\alpha_{\mathfrak{k}}+
\lambda\alpha_{\mathfrak{p}}'',$  $\lambda\in S^1.$
It is well-known that (\cite{DPW}) the map  $\mathcal{F}:M\rightarrow G/K$ is harmonic if and only if
\begin{equation}\label{integr}\dd
\alpha_{\lambda}+\frac{1}{2}[\alpha_{\lambda}\wedge\alpha_{\lambda}]=0,\ \ \hbox{for all}\ \lambda \in S^1.
\end{equation}
 From this for $\mathcal{F}$ we infer that on $U$ there exists a unique solution $F(z,\bar z,\lambda)$  to the equation $\dd F(z,\bar z,\lambda)= F(z,\bar z,\lambda)\alpha_{\lambda},\ F(z_0,\bar z_0,\lambda)=F(z_0,\bar z_0)\in K,~z_0 \in U$.  The solution   $F(z, \lambda)$
is called the {\em extended frame} of the harmonic map $\mathcal{F}$ (normalized at the base point $z=z_0$ if $F(z_0,\bar z_0,\lambda)=e$).
 Then  $\mathcal{F}_{\lambda}:=F(z,\bar z,\lambda)\mod K$ is harmonic for every $\lambda \in S^1$ and this family of harmonic maps will be called the ``associated family'' of the harmonic map
$\mathcal{F}.$

Recall that the twisted loop groups of $G$ and $G^{\mathbb{C}}$  are defined as follows:
\begin{equation*}
\begin{array}{llll}
\Lambda G^{\mathbb{C}}_{\sigma} &=\{\gamma:S^1\rightarrow G^{\mathbb{C}}~|~ ,\
\sigma \gamma(\lambda)=\gamma(-\lambda),\lambda\in S^1  \},\\[1mm]

\Lambda G_{\sigma}  &=\{\gamma\in \Lambda G^{\mathbb{C}}_{\sigma}
|~ \gamma(\lambda)\in G, \hbox{for all}\ \lambda\in S^1 \},\\[1mm]
\Lambda_{*}^{-} G^{\mathbb{C}}_{\sigma}  &=\{\gamma\in \Lambda G^{\mathbb{C}}_{\sigma}~
|~ \gamma \hbox{ extends holomorphically to the domain } |\lambda|>1,\  \gamma(\infty)=e \},\\[1mm]

\Lambda^{+} G^{\mathbb{C}}_{\sigma}  &=\{\gamma\in \Lambda G^{\mathbb{C}}_{\sigma}~
|~ \gamma \hbox{ extends holomorphically to the disk} \hspace{1mm} |\lambda|>1,  \},\\

\Lambda_L^{+} G^{\mathbb{C}}_{\sigma}  &=\{\gamma\in \Lambda G^{\mathbb{C}}_{\sigma}~
|~ \gamma \hbox{ extends holomorphically to the disk} \hspace{1mm} |\lambda|>1,  \gamma(0)\in L\},\\

\end{array}\end{equation*}
where $L\subset K^{\C}$ is a subgroup.  For the loop groups related to Willmore surfaces \cite{DoWa1}, the  Iwasawa decomposition states \cite{DPW,DoWa1} that there exists a closed, connected solvable subgroup $S \subseteq K^\C$ such that
the multiplication $\Lambda G_{\sigma}^{\circ} \times \Lambda^{+}_S G^{\mathbb{C}}_{\sigma}\rightarrow
\Lambda G^{\mathbb{C}}_{\sigma}$ is a real analytic diffeomorphism onto the open subset
$ \Lambda G_{\sigma}^{\circ} \cdot \Lambda^{+}_S G^{\mathbb{C}}_{\sigma}  \subset(\Lambda G^{\mathbb{C}}_{\sigma})^{\circ}$.
The  Birkhoff decomposition states that
the multiplication $\Lambda_{*}^{-} {G}^{\mathbb{C}}_{\sigma}\times
\Lambda^{+}_{ K^{\C}} {G}^{\mathbb{C}}_{\sigma}\rightarrow
\Lambda {G}^{\mathbb{C}}_{\sigma}$ is an analytic  diffeomorphism onto the
open and dense subset $\Lambda_{*}^{-} {G}^{\mathbb{C}}_{\sigma}\cdot
\Lambda^{+}_{ K^{\C}} {G}^{\mathbb{C}}_{\sigma}$ {\em ( big Birkhoff cell )}.

Let $\D$ be  the unit disk or $\C$ itself, with complex coordinate $z$. The DPW construction can be stated as follows
\cite{DPW}, \cite{DoWa1}, \cite{Wu}.
\begin{enumerate}
\item
Let $\mathcal{F}: \D \rightarrow G/K$ be a harmonic map with an extended frame $F(z,\bar z,\lambda)$ satisfying  $F(0,0,\lambda)=e$.
Then there exists a Birkhoff decomposition
$F(z,\bar z,\lambda)=F_-(z,\lambda)  F_+(z,\bar z,\lambda)$ with
$ F_+(z,\bar z,\lambda):\D \rightarrow\Lambda^{+}_{\FC} G^{\mathbb{C}}_{\sigma},$
such that
 $ F_-(z,\lambda):\D \rightarrow\Lambda_{*}^{-} G^{\mathbb{C}}_{\sigma}
$ is meromorphic in $z $ on $ \D$ and satisfies $F_-(0, \lambda) = e.$ Moreover, its Maurer-Cartan form is of the form
\[\eta= F_-(z,\lambda)^{-1} \dd F_-(z,\lambda)=\lambda^{-1}\eta_{-1}(z)\dd z\]
with $\eta_{-1}(z)$ independent of $\lambda$. $\eta$ is called the {\bf normalized potential} of $\mathcal F$.

\item Conversely,  Let $\eta$ be a $\lambda^{-1}\cdot\mathfrak{p}^{\mathbb{C}}-\hbox{valued}$ meromorphic $(1,0)$--form. Let $F_-(z,\lambda)$ be a solution to $F_-(z,\lambda)^{-1} \dd F_-(z,\lambda)=\eta, ~ F_-(z_0,\lambda)=e$
which is meromorphic on $\D$. Then on the open subset $\D_{\mathcal{I}}$  of $\D$  consisiting of all points in $\D$ which are not poles of $F_-$, we have by the Iwasawa decomposition
\[F_-(z,\lambda)=\tilde{F}(z,\bar z,\lambda) \tilde{F}_+(z,\bar z,\lambda)^{-1}, \]
$\hbox{with } \tilde{F}(z,\bar z,\lambda)\in\Lambda G_{\sigma},\  \tilde{F}_+(z,\bar z,\lambda)\in\Lambda^{+}_\FC {G}^{\mathbb{C}}_{\sigma},$ $\tilde{F}(0,0,\lambda) = e$ and $ \tilde{F}_+(0,0, \lambda) = e.$
 Then $\tilde{F}(z,\bar z,\lambda) $ is an extended frame  of some harmonic map from
 $\D_{\mathcal{I}}$  to $G/K$ satisfying   $\tilde{F}(0,0,\lambda)=e$. Moreover, the two constructions above are inverse to each other.
\end{enumerate}

\subsection{Burstall-Guest theroy for harmonic maps of finite uniton type in terms of DPW}

It is well-known due to Uhlenbeck \cite{Uh}, Burstall and Guest \cite{BuGu} that all harmonic two spheres into semi-simple compact Lie groups are of finite uniton type (we refer to \cite{Uh}, \cite{BuGu}, \cite{Gu2002} and \cite{DoWa2} for more details). Moreover, this statement stay true for harmonic maps two spheres into inner, compact symmetric spaces \cite{BuGu}. In \cite{DoWa2}, it was shown to stay true for two spheres
harmonic maps two spheres into inner, non--compact symmetric spaces.

Furthermore, in \cite{BuGu}, \cite{Gu2002}, a way to derive such harmonic maps two spheres are provided. In terms of DPW method, it can be stated as follows.

Now let us turn to the  Burstall-Guest theory for harmonic maps  of finite uniton type into the symmetric space related to Willmore surfaces. Let $U/U\cap K^{\C}$ be the compact dual of $G/K$ \cite{DoWa2}. Let $\mathrm{T}\subset  U\cap{K}^{\C}$ be some maximal torus of $G^{\mathbb{C}}$  with $\mathfrak{t}$ the Lie algebra of $\mathrm{T}$. For any $\gamma_{\xi}\in \mathcal{I}:=(2\pi)^{-1}\exp^{-1}(e)\cap\mathfrak{t}$, define
\begin{equation}\gamma_{\xi}(\lambda):=\exp(t\xi),\ \ \hbox{ for all }\ \ \lambda=e^{it}\in S^1.
\end{equation}
Note that $\gamma_{\xi}(\pi)^2=\gamma_{\xi}(2\pi)=e$.
Let $C_0$ be a fundamental Weyl chamber of $\mathfrak{t}$. Set $\mathcal{I}'=C_0\cap \mathcal{I}$. Then $\mathcal{I}'$ parameterizes the
conjugacy classes of homomorphisms $S^1\rightarrow G$.
Let $\Delta$ be the set of roots of $\mathfrak{g}^{\mathbb{C}}$. Decompose $\Delta$ as $\Delta=\Delta^-\cup\Delta^+$ according to $C_0$.
 Let $\theta_1,\cdots,\theta_l\in \Delta^{+}$ be the simple roots. We denote by $\xi_1,\cdots,\xi_l\in \mathfrak{t}$ the basis of $\mathfrak{t}$ which is dual to
$\theta_1,\cdots,\theta_l$ in the sense that $\theta_j(\xi_k)=\sqrt{-1}\delta_{jk}$.
By \cite{BuGu} (page 555), an element $\xi$ in $\mathcal{I}'\backslash\{0\}$ is called a {\em canonical} element, if $\xi=\xi_{j_1}+\cdots+\xi_{j_k}$ with $\xi_{j_1},\cdots,\xi_{j_k}\in \{\xi_1,\cdots,\xi_l\}$ pairwise different.
For $\theta\in\Delta$ and $X\in\mathfrak{g}_{\theta}$ we obtain
$ad\xi X=\theta(\xi)X\ \ \hbox{ and }\ \theta(\xi)\in\sqrt{-1}\mathbb{Z}.$
Let $\mathfrak{g}^{\xi}_{j}$ be the $\sqrt{-1}\cdot j-\hbox{eigenspace}$ of $\hbox{ad}\xi$. Then
 \begin{equation}
\mathfrak{g}^{\xi}_j=\underset{\theta(\xi)=\sqrt{-1} j}{\oplus}\mathfrak{g}_{\theta}, \hbox{ and } \mathfrak{g}^{\mathbb{C}}=\underset{j}{\oplus}\ \mathfrak{g}^{\xi}_j.
\end{equation}

\begin{theorem}\label{thm-BG}  Theorem 4.13, 4.18 and Theorem 5.3 of \cite{DoWa2}, compare also \cite{BuGu} and \cite{Gu2002}.
\begin{enumerate}
  \item Let $\mathcal{F}:M\rightarrow G/K$ be a harmonic map of finite uniton number.  Then there exists some canonical $\xi\in \mathcal{I}'$,
some discrete subset $D'\subset M$, such that
$G/K\cong \{g(\exp \pi \xi)g^{-1}| g\in G\}$
and  $F_-:=\gamma_{\xi}^{-1}\exp C \gamma_{\xi}$ is a meromorphic extended frame of $\mathcal{F}$ with the normalized potential having the form
\begin{equation}\eta=F_-^{-1}\mathrm{d}F_-=\lambda^{-1}\eta_{-1}\dd z=\lambda^{-1}\sum_{0\leq 2j\leq r(\xi)-1} A_{2j}'\mathrm{d}z, ~ F_-(z_0,\bar{z}_0,\lambda)=e \hbox{ for some }z_0\in M, \end{equation}
where $A_{2j}':M\setminus D'\rightarrow \mathfrak{g}^{\xi}_{2j+1}$ is a {\em meromorphic map}.
\item  Conversely, given a  meromorphic normalized potential $\eta$ taking values in  $\lambda^{-1}\cdot\sum_{0\leq 2j\leq r(\xi)-1}\mathfrak{g}^{\xi}_{2j+1}$,
 the Iwasawa decomposition of the solution $F_-$ to
 \begin{equation}\label{eq-initial3}F_-^{-1}\mathrm{d}F_-=\eta,~ F_-(z_0,\bar{z}_0,\lambda)=e \hbox{ for some }z_0\in M, \end{equation}
{gives the extended frame of a harmonic} map of finite uniton number into $ G/ K $.

 \item Let $\mathcal{F}:S^2\rightarrow G/K$ be a harmonic map. Then it is of finite uniton number.
\end{enumerate}
\end{theorem}

\begin{remark}\
\begin{enumerate}
  \item
It is obvious that $\eta_{-1}$ takes value in $\sum \mathfrak{g}^{\xi}_{2j+1}$, which is in a nilpotent Lie subalgebra of $\mathfrak{g}^{\C}$.
This explains what we mean before. Moreover, it is straightforward to see that $F_-$ can be derived by finite many times iterations (See for example \cite{Gu2002}).
The Iwasawa decomposition of  $F_-$ can also be done in finite procedures. Therefore this provides an explicit way to obtain such harmonic maps. Furthermore, some geometric properties of the harmonic maps can also be derived this way.  See for instance \cite{Gu2002}, \cite{Wang-2} and \cite{Wang-3} for concrete illustrations.

\item From this theorem, one can derive all normalized potentials related to the harmonic maps of finite uniton type, by classifying all the  nilpotent Lie subalgebras related to the given symmetric space. To classify these nilpotent Lie  subalgebras, it suffices to classify the corresponding $\xi$. These will be the main content of the next sections.

\end{enumerate}\end{remark}

 \section{ Coarse classification of Willmore surfaces of finite uniton type }

Now we consider  strongly conformally harmonic maps of finite uniton type in  $SO^+(1,n+3)/SO^+(1,3)\times SO(n)$.
Due to the seminal work of \cite{BuGu} (see also \cite{Gu2002}), and also Theorem 5.3 of \cite{DoWa2},  harmonic maps of finite uniton type into any inner symmetric space, compact or non-compact, can be described by some nilpotent Lie subalgebra valued meromorphic 1-forms. Applying this to strongly conformally harmonic maps of finite uniton type in  $SO^+(1,n+3)/SO^+(1,3)\times SO(n)$,  the nilpotent Lie subalgebras need to be related to the symmetric space $SO^+(1,n+3)/SO^+(1,3)\times SO(n)$ \cite{BuGu}. This will provide the forms of the (nilpotent) normalized potentials  of strongly conformally harmonic maps of finite uniton type. Moreover, the conditions on $B_1$ will give further restrictions. Altogether, we will obtain a description of the normalized potentials of  strongly conformally harmonic maps of finite uniton type, and hence  the  normalized potentials of  Willmore surfaces of finite uniton type. Especially, we obtain a description of the normalized potentials of all Willmore two-spheres. To be concrete, we have the following theorem.

 \subsection{The general case}  Let $\D$ be a contractible open subset of $\C$ with complex coordinate $z$.
\begin{theorem}\label{thm-potential}
Let $\mathcal{F}:\D\rightarrow SO^+(1,n+3)/SO^+(1,3)\times SO(n)$ be a strongly conformally harmonic map of finite-uniton type, with $n+4=2m$. Then  up to  conjugation of some matrix in $O^+(1,3)\times O(n)$, there exists a normalized potential
\[\eta=\lambda^{-1}\left(
                     \begin{array}{cc}
                       0 & \hat{B}_1 \\
                       -\hat{B}^{t}_1I_{1,3} & 0 \\
                     \end{array}
                   \right)\mathrm{d}z \]
                 of $\mathcal{F}$  such that
                 \[\hat{B}_1=\left(\hat{B}_{13},\cdots,\hat{B}_{1m}\right),\ \hat{B}_{1j}=\left(\mathrm{v}_{j },\hat{ \mathrm{v}}_{j }\right)\in Mat(4\times 2,\mathbb{C}),\]
and every $\hat{B}_{1j}$ of $\hat{B}_1$ has one of the following two forms:
\begin{equation}(i)~~~ (\mathrm{v}_{j},\hat{ \mathrm{v}}_{j }) = \left(
                                          \begin{array}{ccccccc}
                                            h_{1j} & \hat{h}_{1j}   \\
                                            h_{1j} & \hat{h}_{1j}  \\
                                            h_{3j} & \hat{h}_{3j} \\
                                            ih_{3j}& i\hat{h}_{3j} \\
                                          \end{array}
                                        \right);\ \ \
 (ii) ~~~ (\mathrm{v}_{j},\hat{ \mathrm{v}}_{j }) = \left(
                                                             \begin{array}{cc}
                                                               h_{1j} & ih_{1j} \\
                                                               h_{2j} & ih_{2j}  \\
                                                               h_{3j}  & ih_{3j}  \\
                                                               h_{4j}  & ih_{4j}  \\
                                                             \end{array}
                                                           \right)
\end{equation}
with all of $\{\mathrm{v}_j,\  \hat{ \mathrm{v}}_{j }\}$ satisfying the following conditions
\begin{equation}\mathrm{v}_j^tI_{1,3}\mathrm{v}_l=\mathrm{v}_j^tI_{1,3}\hat{\mathrm{v}}_l=\hat{\mathrm{v}}_j^tI_{1,3}\hat{\mathrm{v}}_l=0, \ j,l=3,\cdots,m.
\end{equation}
In other words, there are $m-1$ types of normalized potentials with $\hat{B}_{1}$ satisfying $\hat{B}_{1 }^tI_{1,3}\hat{B}_{1 }=0$, namely those being of one of the following $m-1$ forms:

$(1)$ $($all pairs are of type $(i))$

\begin{equation}\label{eq-potential-1}
                         \hat{B}_{1}= \left(
                                          \begin{array}{ccccccc}
                                            h_{13} & \hat{h}_{13} &  h_{14} & \hat{h}_{14} &\cdots &  h_{1m}& \hat{h}_{1m} \\
                                            h_{13} & \hat{h}_{13} &  h_{14}& \hat{h}_{14}&\cdots &   h_{1m}  & \hat{h}_{1m}\\
                                            h_{33} & \hat{h}_{33} &  h_{34}& \hat{h}_{34} &\cdots &  h_{3m} & \hat{h}_{3m} \\
                                            ih_{33}& i\hat{h}_{33} &  ih_{34}& i\hat{h}_{34}&\cdots &  ih_{3m}& i\hat{h}_{3m} \\
                                          \end{array}
                                        \right);\end{equation}

$(2)$ $($the first pair is of type $(ii)$, all others are of type $(i))$
\begin{equation}\label{eq-potential-2}
                           \hat{B}_{1}= \left(
                                          \begin{array}{ccccccc}
                                            h_{13}& i {h}_{13} &  h_{14}& \hat{h}_{14} &\cdots &  h_{1m}&  \hat{h}_{1m}\\
                                            h_{23}& i {h}_{23} &  h_{14}& \hat{h}_{14} &\cdots &  h_{1m}&  \hat{h}_{1m}\\
                                            h_{33}& i {h}_{33} &  h_{34}& \hat{h}_{34} &\cdots &  h_{3m}&  \hat{h}_{3m}\\
                                            h_{44}& i {h}_{43} & ih_{34}&i\hat{h}_{34} &\cdots & ih_{3m}& i\hat{h}_{3m}\\
                                          \end{array}
                                        \right);\end{equation}
Introducing consecutively  more pairs of type $(ii)$, one finally arrives at

       $(m-1)$ $($all pairs are of type $(ii)$$)$
\begin{equation}\label{eq-potential-m}
                           \hat{B}_{1}= \left(
                                          \begin{array}{ccccccc}
                                            h_{13}& i {h}_{13} &  h_{14} & i {h}_{14} &\cdots &  h_{1m} & ih_{1m} \\
                                            h_{23}& i {h}_{23} &  h_{24} & i {h}_{24} &\cdots &  h_{2m} & ih_{2m}\\
                                            h_{33}& i {h}_{33} &  h_{34} & i {h}_{34} &\cdots &  h_{3m} & ih_{3m} \\
                                            h_{43}& i {h}_{43} &  h_{44} & i {h}_{44} &\cdots &  h_{4m} & ih_{4m} \\
                                          \end{array}
                                        \right).\end{equation}

Moreover, if $\mathcal{F}$ is the conformal Gauss map of some Willlmore map from $S^2$ to $S^{2m+2}$, then $\mathfrak{F}$ is of finite uniton type and hence its normalized potential is one of the above forms.
\end{theorem}

\begin{proof}
By Theorem \ref{thm-BG}, as discussed before, the proof of Theorem \ref{thm-potential} reduces to classifying  nilpotent Lie sub-algebras related to $SO^+(1,n+3)/SO(1,3)\times SO(n)$ and to find out the potentials taking values in these nilpotent Lie sub-algebras which are related to Willmore surfaces \cite{BuGu, DoWa2}. Since the detailed computations are lengthy and technical, we will divide them into several lemmas in Section 3, and leave the concrete proof of this theorem to Section 3.
 \end{proof}
Note that both type $(1)$ and type $(m-1)$ are of finite uniton number $2$ ( Lemma 3.5 of Section 3.2) and are hence $S^1-$invariant by Corollary 5.6 of \cite{BuGu} (see also \cite{Do-Es} for $S^1-$invariant harmonic maps), which also provides a proof of  Corollary 5.10, Corollary 5.11 and part of Theorem 5.3  of \cite{DoWa1}. The remaining cases are in general of finite uniton number $\geq4$ and hence will be more complicated.
 Here we will  use Theorem \ref{thm-potential} to derive some geometric properties of these maps and list some new examples.

\begin{remark}\

  \begin{enumerate}
  \item If $rank(\hat{B}_1)=1$, the potentials of type (1) provide  no Willmore maps, or Willmore maps with a constant light-like vector, which turn out to be corresponding to minimal surfaces in $\mathbb{R}^{n+2}$.  When $rank(\hat{B}_1)=2$ the potentials of the first type will produce no Willmore surfaces at all \cite{Wang-2}.
 \item Potentials of type ($m-1$) are conjectured to correspond to totally isotropic Willmore surfaces. This has been proven in Theorem 3.1 of \cite{Wang-3} when $m=4$, i.e., for Willmore two-spheres in $S^6$. Moreover, excluding the intersection with the first type, every normalized potential of this type produces a unique non S--Willmore surface when $rank(\hat{B}_1)=2$ and gives a pair of dual Willmore surfaces when $rank(\hat{B}_1)=1$.
\item
For any potential of the other types, excluding the intersections with the first type and the last type, one obtains a unique non S--Willmore surface which has a non-isotropic conformal Hopf differential, since in such cases $rank(\hat{B}_1)=2$.
\end{enumerate}
\end{remark}

 \subsection{Willmore two-spheres in $S^6$}
Concerning the case of Willmore two-spheres in $S^6$, we have a geometric description as follows:
\begin{theorem}\label{thm-potential-s6}Let $y:S^2\rightarrow S^6$ be a strong Willmore map. Assume that the normalized potential of (the conformal Gauss map of) $y$ is of the form
\[\eta=\lambda^{-1}\left(
                      \begin{array}{cc}
                        0 & \hat{B}_1 \\
                        -\hat{B}_1^tI_{1,3} & 0 \\
                      \end{array}
                    \right)\mathrm{d}z=\lambda^{-1}\eta_{-1}\mathrm{d}z.\]
                    Then $y$, as well as $\eta$,  belongs to one of the three cases:

$(1)$ $y$ is conformally equivalent to  a complete minimal surface in $\R^6$ with planar ends. In this case,  up to  conjugation of some matrix in $O^+(1,3)\times O(n)$, $\hat{B}_1$ is of the form
   \begin{equation}\label{eq-minimal} \hat{B}_1=
                                        \left(
                                          \begin{array}{ccccc}
                                             \mathrm{v}_1 & h_{20}\mathrm{v}_1 & h_{30}\mathrm{v}_1  &  h_{40}\mathrm{v}_1  \\
                                          \end{array}
                                        \right)~~\hbox{ with }~~\mathrm{v}_1 =
                                        \left(
                                          \begin{array}{c}
                                            \tilde{h}_{13} \\
                                            \tilde{h}_{13} \\
                                            \tilde{h}_{33} \\
                                           i\tilde{h}_{33}  \\
                                          \end{array}
                                        \right)   \  \hbox{ and } \ | \tilde{h}_{33}'|^2 \not\equiv 0
       .                                  \end{equation}

$(2)$  $y$ is not S--Willmore and the Hopf differential of $y$ is not isotropic.
In this case,   up to  conjugation of some matrix in $O^+(1,3)\times O(n)$, $\hat{B}_1$ is of the form
   \begin{equation}\label{eq-nm-ni} \hat{B}_1=
                                        \left(
                                          \begin{array}{ccccc}
                                              h_{10}\mathrm{v}_1 & i h_{10}\mathrm{v}_1 & \hat{h}_{30}\mathrm{v}_2  &  \hat{h}_{40}\mathrm{v}_2  \\
                                          \end{array}
                                        \right)\     \end{equation}
with   \begin{equation}\label{eq-iso-v} \mathrm{v}_1=
                                        \left(
                                          \begin{array}{c}
                                           1+ h_1h_2    \\
                                           -1+ h_1h_2   \\
                                           h_1+h_2  \\
                                           -i(h_1-h_2)   \\
                                          \end{array}
                                        \right)~~\hbox{ and }~~  \mathrm{v}_2 =
                                        \left(
                                          \begin{array}{c}
                                             h_1 \\
                                            h_1 \\
                                             1 \\
                                             i  \\
                                          \end{array}
                                        \right)
                                        \end{equation}
and \begin{equation}\label{eq-non-deg2}(|h_1'|^2+|h_2'|^2) (\hat{h}_{30}^2+ \hat{h}_{40}^2)\not\equiv 0. \end{equation}

$(3)$ $y$ is  totally isotropic in $S^6$, i.e, it comes from the twistor projection of some holomorphic or anti-holomorphic curve  into the twistor bundle $\mathfrak{T}S^6$ of $S^6$. In this case,    up to  conjugation of some matrix in $O^+(1,3)\times O(n)$, $\hat{B}_1$ is of the form
   \begin{equation}\label{eq-isotropic} \hat{B}_1=
                                        \left(
                                          \begin{array}{ccccc}
                                             h_{10}\mathrm{v}_1 & ih_{10}\mathrm{v}_1 & h_{30}\mathrm{v}_1+h_{40}\mathrm{v}_2   &  i(h_{30}\mathrm{v}_1+h_{40}\mathrm{v}_2) \\
                                          \end{array}
                                        \right)\
                                        \end{equation}
                                        with  $\mathrm{v}_1$ and $\mathrm{v}_2$ of the form \eqref{eq-iso-v}.
                                       and \begin{equation}\label{eq-non-deg3}| h_1'|^2(|h_{30}|^2+|h_{40}|^2)\not\equiv 0. \end{equation}

   In all of above cases, $h_{10}$, $h_{20}$, $h_{30}$, $h_{40}$, $\hat{h}_{30}$, $\hat{h}_{40}$, $h_1$ and $h_2$ are assumed to be meromorphic functions on $S^2$ such that $\eta$ has a global meromorphic framing on $S^2$, i.e., there exists some meromorphic framing $F_-:S^2\rightarrow \Lambda^- G^{\C}_{\sigma}$ satisfying $F_-^{-1}\mathrm{d} F_-=\eta .$
   \end{theorem}
\begin{proof} Setting $m=4$ in Theorem \ref{thm-potential}, we see that there are three kinds of possible nilpotent normalized potentials from \eqref{eq-potential-1}, \eqref{eq-potential-2} and \eqref{eq-potential-m}:
\begin{equation}\label{eq-potent6-1}
                         \hat{B}_{1}= \left(
                                          \begin{array}{ccccccc}
                                            h_{13} &  \hat{h}_{13} &   h_{14} &  \hat{h}_{14} \\
                                            h_{13} &  \hat{h}_{13} &   h_{14} &  \hat{h}_{14} \\
                                            h_{33} &  \hat{h}_{33} &   h_{34} &  \hat{h}_{34} \\
                                           ih_{33} & i\hat{h}_{33} &  ih_{34} & i\hat{h}_{34} \\
                                          \end{array}
                                        \right),\end{equation}
                         \begin{equation} \label{eq-potent6-2}
                           \hat{B}_{1}= \left(
                                          \begin{array}{ccccccc}
                                            h_{13} & i {h}_{13} &  h_{14} &  \hat{h}_{14}   \\
                                            h_{23} & i {h}_{23} &  h_{14} &  \hat{h}_{14} \\
                                            h_{33} & i {h}_{33} &  h_{34} &  \hat{h}_{34}  \\
                                            h_{43} & i {h}_{43} & ih_{34} & i\hat{h}_{34}  \\
                                          \end{array}
                                        \right),\end{equation}
and
\begin{equation} \label{eq-potent6-3}
                           \hat{B}_{1}= \left(
                                          \begin{array}{ccccccc}
                                            h_{13} & i {h}_{13} &  h_{14} & i {h}_{14} \\
                                            h_{23} & i {h}_{23} &  h_{24} & i {h}_{24} \\
                                            h_{33} & i {h}_{33} &  h_{34} & i {h}_{34}  \\
                                            h_{43} & i {h}_{43} &  h_{44} & i {h}_{44}  \\
                                          \end{array}
                                        \right).\end{equation}
We see that \eqref{eq-minimal} comes from \eqref{eq-potent6-1}.

We need to show that \eqref{eq-nm-ni} comes from \eqref{eq-potent6-2} and \eqref{eq-isotropic} comes from \eqref{eq-potent6-3}.
 Assume that
$\hat{B}_1=
                                        \left(
                                          \begin{array}{ccccc}
                                             \check{\mathrm{v}}_1 &  \check{\mathrm{v}}_2 &   \check{\mathrm{v}}_3  &    \check{\mathrm{v}}_4  \\
                                          \end{array}
                                        \right)$. Notice that  \eqref{eq-Willmore harmonic} requires
                                        \[\check{\mathrm{v}}_j^t I_{1,3}\check{\mathrm{v}}_k=0,\ \hbox{ for all }\ j,k=1,\cdots,4.
                                        \]
                                        Especially,  $\check{\mathrm{v}}_j^t I_{1,3}\check{\mathrm{v}}_j=0$. So by (17) of \cite{MWW} we can assume that  $ \check{\mathrm{v}}_1=h_{10} \mathrm{v}_1$ and $\check{\mathrm{v}}_2=ih_{10} \mathrm{v}_1$, with
                                         \[
                                       \mathrm{v}_1=
                                        \left(
                                          \begin{array}{c}
                                           1+ h_{1}h_{2}    \\
                                           -1+{h}_{1}{h}_{2}   \\
                                          h_{1}+{h}_{2}  \\
                                          -i(h_{1}-{h}_{2})   \\
                                          \end{array}
                                        \right) \hbox{ and } \mathrm{v}_2 =
                                        \left(
                                          \begin{array}{c}
                                             h_{1} \\
                                            h_{1} \\
                                            1 \\
                                             i  \\
                                          \end{array}
                                        \right).\]
Then the conditions $\check{\mathrm{v}}_j^tI_{1,3}\check{\mathrm{v}}_1=\check{\mathrm{v}}_j^tI_{1,3}\check{\mathrm{v}}_j=0$ and  $\check{\mathrm{v}}_j\in \mathbb{R}^{1,3}\otimes \mathbb{C}$, $j=3,4$, indicates that
\[\check{\mathrm{v}}_j=\check{h}_{j0}\mathrm{v}_1+\tilde{h}_{j0}\mathrm{v}_2, \ j=3,4,\]
 for some functions $\check{h}_{j0}, \tilde{h}_{j0}$. Together with the restrictions of \eqref{eq-potent6-2} and \eqref{eq-potent6-3}, \eqref{eq-nm-ni} and \eqref{eq-isotropic} follow straightforwardly.

 As to the geometric descriptions, Case (1) follows from Theorem 2.1  of \cite{Wang-2} and Case (3) follows from Theorem 3.1 of \cite{Wang-3}.
For Case (2), $y$ is not S--Willmore since $rank(\hat{B}_1)=2$. Assume $\langle\kappa,\kappa\rangle\equiv0$, then $\langle D_{\bar{z}}\kappa,\kappa\rangle\equiv0$,  and
$\langle D_{\bar{z}}\kappa,D_{\bar{z}}\kappa\rangle=-\langle D_{\bar{z}}D_{\bar{z}}\kappa,\kappa\rangle\equiv0$,
since $\kappa$ satisfies the Willmore equation $D_{\bar{z}}D_{\bar{z}}\kappa+\frac{\bar{s}}{2}\kappa=0$. Therefore
$\kappa$ being isotropic is equivalent to $B_1B_1^tI_{1,3}\equiv0$, which is equivalent to $\hat{B}_1\hat{B}^tI_{1,3}\equiv0$ by Wu's formula \cite{Wu}, \cite{DoWa1}. On the other hand  $\hat{B}_1\hat{B}^tI_{1,3}\equiv0$ is equivalent to $h_{30}^2+ h_{40}^2=0$, which is not allowed due to \eqref{eq-non-deg2}. Therefore $\kappa$ is not isotropic.
\end{proof}

\begin{remark}It is usually not easy to check whether $y$  is an immersion or not. Fortunately we do have an example of type (3) which is a non--S--Willmore Willmore immersion. See the next subsection for details.

 \end{remark}

\begin{corollary}   Let $y:S^2\rightarrow S^4$ be a strong Willmore map such that, up to  conjugation of some matrix in $O(1,3)\times  O(2)$, the  normalized potential of its the conformal Gauss map is of the form
\[\eta=\lambda^{-1}\left(
                      \begin{array}{cc}
                        0 & \hat{B}_1 \\
                        -\hat{B}_1^tI_{1,3} & 0 \\
                      \end{array}
                    \right)\mathrm{d}z=\lambda^{-1}\eta_{-1}\mathrm{d}z,\]
                    Then $y$, as well as $\eta$,  belongs to one of the cases:
\begin{enumerate}
\item $y$ is conformally equivalent to  a complete minimal surface in $\R^4$ with planar ends. In this case, $r(\xi)=2$ and
   \begin{equation} \hat{B}_1=
                                        \left(
                                          \begin{array}{ccccc}
                                             h_{10}\mathrm{v}_1 & h_{20}\mathrm{v}_1  \\
                                          \end{array}
                                        \right),\end{equation}
                                          with
                                        \begin{equation}\mathrm{v}_1^t=\left(
                                                                               \begin{array}{cccc}
                                                                                 h_1&
                                                                                 h_1 &                                                                                 ih_{2} &
                                                                                 ih_{2} \\
                                                                               \end{array}
                                                                             \right)
\hbox{ and } |h_1|^2+| h_2|^2  \not\equiv 0 .    \end{equation}
 \item $y$  is  isotropic in $S^4$, therefore it comes from the twistor projection of a holomorphic or anti-holomorphic curve into the twistor bundle $\C P^3$ of $S^4$. In this case, $r(\xi)=2$ and
   \begin{equation} \hat{B}_1=
                                        \left(
                                          \begin{array}{ccccc}
                                             h_{10}\mathrm{v}_1 & ih_{10}\mathrm{v}_1  \\
                                          \end{array}
                                        \right)\
                                        \end{equation}
                                        with  $\mathrm{v}_1$  being of the form \eqref{eq-iso-v}.
\end{enumerate}
   In all of the above cases, $h_{10}$, $h_{20}$, $h_1$ and $h_{2}$ are meromorphic functions on $S^2$ such that the solution to $F_-^{-1}\mathrm{d}F_-=\eta$, $F_-(0,\lambda)=I_6$ is meromorphic on $S^2$.

   In particular, both kinds of Willmore maps as above are S--Willmore.
\end{corollary}

\begin{proof} Restricting to $S^4$, it is easy to see that Case (2) of Theorem \ref{thm-potential-s6} can not happen. The isotropic case is a corollary of Theorem \ref{thm-potential-s6}. The last statement comes from the simple observation that in both cases, $rank\hat{B}_1=1$.
\end{proof}

 The maps of type (3) in Theorem \ref{thm-potential-s6} are full in some even dimensional spheres. And the maps of type (2) sometimes can reduce to maps into some $S^5\subset S^6$.
 Therefore we obtain
 \begin{corollary} Let $y:S^2\rightarrow S^5\subset S^6$ be a Willmore map, which is not S--Willmore. Then, up to  conjugation of some matrix in $O(1,3)\times  O(3)$, the normalized potential of $y$ is of the form
\[\eta=\lambda^{-1}\left(
                      \begin{array}{cc}
                        0 & \hat{B}_1 \\
                        -\hat{B}_1^tI_{1,3} & 0 \\
                      \end{array}
                    \right)\mathrm{d}z, \]with\begin{equation}\label{eq-s5}
\hat{B}_{1}= \left(
                                          \begin{array}{ccccccc}
                                              h_0( 1+ h_{1}h_{2} ) & i h_0( 1+ h_{1}h_{2} )     &  \hat{h}_0h_{1} & 0 \\
                                             h_0(-1+{h}_{1}{h}_{2} ) & ih_0(-1+{h}_{1}{h}_{2} ) & \hat{h}_0h_{1}&0\\
                                              h_{0}(h_{1}+{h}_{2})    & ih_{0}(h_{1}+{h}_{2})   &  \hat{h}_0 &0 \\
                                              -i h_{0}(h_{1}-{h}_{2})    & h_{0}(h_{1}-{h}_{2}) & i\hat{h}_0 &0\\
                                          \end{array}
                                        \right)
\end{equation}
with $h_j$, $j=0,1,2$, and $\hat{h}_0$ being non-constant meromorphic functions on $S^2$. Note that $\hat{B}_1$ is of type (2) in Theorem \ref{thm-potential-s6} (excluding type (1) and type (3) in Theorem \ref{thm-potential-s6}).
\end{corollary}
If one chooses the $h_j$ as polynomials in $z$, then one will obtain a possibly branched Willmore two-sphere in $S^5$ which is not S--Willmore.

 In \cite{MWW2}, a Willmore two-sphere in $S^5$ which is not S--Willmore is provided. This example is given by some adjoint transforming of a minimal surface in $\R^5$ with isotropic Hopf differential and special ends \cite{MWW2}. Moreover, applying the classification theorem in \cite{MWW2}, one obtains
  \begin{corollary} The Willmore 2-spheres in $S^5$ obtained in the corollary above can be derived by some adjoint transform of a minimal surface in $\R^5$ with isotropic Hopf differential and special ends.
\end{corollary}

Note that although we aim to classify Willmore two--spheres, the above procedures also can be used to tori or other Riemann surfaces to obtain Willmore maps with non-trivial topology.
 \begin{example} Examples of non S--Willmore branched Willmore tori in $S^6$ of finite uniton type.
 \begin{enumerate}
 \item  Let $\eta$ be a normalized potential with $\hat{B}_1$ of the form (type (3) in Theorem \ref{thm-potential-s6})
 \[\hat{B}_{1}= \left(
                                          \begin{array}{ccccccc}
                                            h_{1}h_0 & ih_{1}h_0  &  h_{2} & ih_2 \\
                                             -h_{1}h_0 & -ih_{1}h_0  &  h_{2}& ih_2\\
                                            h_{1}& i{h}_{1} &  h_{2}h_0 & ih_2h_0 \\
                                            -ih_{1}& {h}_{1} & i h_{2}h_0 & -h_2h_0\\
                                          \end{array}
                                        \right)\]
with $h_j$, $j=0,1,2$, non-constant meromorphic functions on $T^2$. Moreover, if the following six integrals
  \begin{equation}\label{eq-t2}
    \begin{split}
&  \mathbf{h}_{1}=\int h_1\mathrm{d}z,\  \mathbf{h}_{2}=\int h_2\mathrm{d}z,\ \mathbf{h}_{10}=\int h_1h_0 \mathrm{d}z,  \ \mathbf{h}_{20}=\int h_2h_0 \mathrm{d}z,\\
 &  \mathbf{h}_{31}=-\int\mathbf{h}_{10}h_2\mathrm{d}z+\int\mathbf{h}_{1}h_2h_0\mathrm{d}z,  \mathbf{h}_{32}= -\int  \mathbf{h}_{2}h_1h_0\mathrm{d}z+\int  \mathbf{h}_{20}h_1\mathrm{d}z
    \end{split}
\end{equation}
are also meromorphic functions on $T^2$, one will obtain a totally isotropic strong Willmore map with trivial monodromy which is full in $S^6$.

 \item It is not difficult to satisfy the conditions above. For example, let $\hat{h}_1,\hat{h}_2$ be meromorphic functions on $T^2$.  Then $h_1=h_2=\hat{h}_1'$, $h_0=\frac{\hat{h}_2'}{\hat{h}_1'}$ will provide a totally isotropic strong Willmore map from $T^2$ to $S^6$.
  However, to determine when such a Willmore map will be an immersion will be a highly non-trivial and interesting problem, see for example \cite{Bryant1982, Bryant1984, Bryant1988, Kusner1989}. Note that solving similar problems is also an interesting topic in the minimal surface theory.
\end{enumerate}
\end{example}

\subsection{A concrete example}

To derive concrete examples, one needs to work out the Iwasawa decompositions in a concrete fashion. We will leave it to \cite{Wang-3}. Here we will just show one example of Willmore 2-sphere which is fully immersed in to $S^6$. This is the first known example of Willmore two-sphere admitting no dual surface.

\begin{theorem}\label{thm-example}\cite{DoWa1}, \cite{Wang-3} Let \begin{equation}\label{eq-example-np}\eta=\lambda^{-1}\left(
                      \begin{array}{cc}
                        0 & \hat{B}_1 \\
                        -\hat{B}_1^tI_{1,3} & 0 \\
                      \end{array}
                    \right)\mathrm{d}z,\ \hbox{ with } \ \hat{B}_1=\frac{1}{2}\left(
                     \begin{array}{cccc}
                       2iz&  -2z & -i & 1 \\
                       -2iz&  2z & -i & 1 \\
                       -2 & -2i & -z & -iz  \\
                       2i & -2 & -iz & z  \\
                     \end{array}
                   \right).\end{equation}
Then $\hat{B}_1$ is of type (3) in Theorem \ref{thm-potential-s6} and the associated family of unbranched Willmore two-spheres $x_{\lambda}$, $\lambda\in S^1$, corresponding to $\eta$, is \begin{equation}\label{example1}
 x_{\lambda} =\frac{1}{ \left(1+r^2+\frac{5r^4}{4}+\frac{4r^6}{9}+\frac{r^8}{36}\right)}
\left(
                          \begin{array}{c}
                            \left(1-r^2-\frac{3r^4}{4}+\frac{4r^6}{9}-\frac{r^8}{36}\right) \\
                            -i\left(z- \bar{z})(1+\frac{r^6}{9})\right) \\
                            \left(z+\bar{z})(1+\frac{r^6}{9})\right) \\
                            -i\left((\lambda^{-1}z^2-\lambda \bar{z}^2)(1-\frac{r^4}{12})\right) \\
                            \left((\lambda^{-1}z^2+\lambda \bar{z}^2)(1-\frac{r^4}{12})\right) \\
                            -i\frac{r^2}{2}(\lambda^{-1}z-\lambda \bar{z})(1+\frac{4r^2}{3}) \\
                            \frac{r^2}{2} (\lambda^{-1}z+\lambda \bar{z})(1+\frac{4r^2}{3})  \\
                          \end{array}
                        \right)
\end{equation}
with $r=|z| .$ Note that for every  $\lambda\in S^1$, $x_{\lambda}$ is isometric to the other ones by some rotation in $SO(7)$.  Moreover $x_{\lambda}:S^2\rightarrow S^6$ is a Willmore immersion in $S^6$, which is full, not S--Willmore, and totally isotropic.
\end{theorem}

\begin{remark} It has been shown in \cite{Ejiri1988} (see also \cite{Mon}, \cite{Mus1}) that there exist S--Willmore  two-spheres which are obtained by the twistor projection of holomorphic or anti-holomorphic curves in the twistor bundle $\mathfrak{T}S^{2n}$ of $S^{2n}$ (for a general theory about twistor geometry, we refer to \cite{BR}). Our example shows that Willmore two-spheres derived in this way can  also be non--S--Willmore. And we also note that even in $S^6$, in general, the surfaces obtained by the twistor projection of holomorphic curves of the twistor bundle $\mathfrak{T}S^{6}$ will not be Willmore.
\end{remark}
\subsection{Further discussions on general cases}
Applying the discussions in Theorem \ref{thm-potential-s6}, we obtain immediately the following more explicit descriptions of the potentials in Theorem \ref{thm-potential}.
\begin{theorem}\label{thm-potential-1}
We retain the notations in Theorem \ref{thm-potential}. Then one of the following three cases happens
\begin{enumerate}
\item $\hat B_1$ is of the form $(1)$ in Theorem \ref{thm-potential};
\item there exists $\mathrm{v}_1$ and $\mathrm{v}_2$ of the form \eqref{eq-iso-v} such that
\[\mathrm{v}_{j}=-i\hat{\mathrm{v}}_j=h_{j0}\mathrm{v}_1+\tilde{h}_{j0}\mathrm{v}_2,\ 3\leq j\leq l;
\mathrm{v}_{j}=\tilde{h}_{j0}\mathrm{v}_2,\ \hat{\mathrm{v}}_{j}=\hat{h}_{j0}\mathrm{v}_2, l+1\leq j\leq m;\]
\item there exists $\mathrm{v}_1$ and $\mathrm{v}_2$ of the form \eqref{eq-iso-v} such that
\[\mathrm{v}_{j}=-i\hat{\mathrm{v}}_j=h_{j0}\mathrm{v}_1+\tilde{h}_{j0}\mathrm{v}_2, 1\leq j\leq  m.\]
\end{enumerate}
\end{theorem}

\section{Classification of Nilpotent Lie subalgebras in
$SO^+(1,n+3,\C)$}

In this section we will give a proof of Theorem \ref{thm-potential} by describing all nilpotent Lie subalgebras of $\mathfrak{so}(1,n+3,\C)$ which correspond to the symmetric space $SO^+(1,n+3)/SO^+(1,3)\times SO(n)$. We divide this proof into two steps:
\begin{enumerate}
\item Classifying the corresponding Lie subalgebras in $\mathfrak{so}(1,n+3,\C)$ related to $SO^+(1,n+3)/SO^+(1,3)\times SO(n)$;
\item Using the condition $B_1^tI_{1,3}B_1=0$ to obtain the classification of normalized potentials of strongly conformally harmonic maps.
\end{enumerate}
Correspondingly we have two subsections.

\subsection{The nilpotent Lie subalgebras associated with $SO^+(1,n+3)/SO^+(1,3)\times SO(n)$}

By Theorem 4.11 of \cite{DoWa2}, the nilpotent Lie subalgebras in question are in one to one relation with canonical elements $\xi$. So the first step in our classification will be to derive all  canonical elements $\xi$ related to the symmetric space $SO^+(1,n+3)/SO^+(1,3)\times SO(n)$. Then after computation of the spaces $\sum_{j>0}g^{\xi}_j$, we will obtain the desired nilpotent Lie subalgebras.

We assume that $n$ is even and set $n+4=2m$. Recall that
\[\mathfrak{g }=\mathfrak{so}(1,n+3,\C)=\{A\in Mat(2m, \mathbb{C})| A^tI_{1,2m-1}+I_{1,2m-1}A=0 \}.\]
 We choose a maximal torus of with its Lie algebra $\mathfrak{t}\subset\mathfrak{k}\subset\mathfrak{so}^+(1,n+3,\C)$ being of the form
\begin{equation}\label{eq-xi}\mathfrak{t}=\left(
                                 \begin{array}{cccccc}
                                    \mathbf{a}_{1} &  &  \\
                                    & \ddots &     \\
                                    &  &  \mathbf{a}_{m} \\
                                 \end{array}
                               \right),\ \mathbf{a}_{1}=\left(
                                                         \begin{array}{cc}
                                                           0 & i\cdot a_{11} \\
                                                           i\cdot a_{11} & 0 \\
                                                         \end{array}
                                                       \right),\ \mathbf{a}_{j}=\left(
                                                         \begin{array}{cc}
                                                           0 & a_{jj} \\
                                                           -a_{jj} & 0 \\
                                                         \end{array}
                                                       \right), \ j=2,\cdots, m.
               \end{equation}
Denote by $\hat{\xi}_l$  the element in $\mathfrak{t}$ with $a_{jj}=\delta_{jl}$.
 So any $\xi\in \mathfrak{t}$ can be expressed as \[\xi=a_{11}\hat{\xi}_1+\cdots+a_{mm}\hat\xi_m.\] Moreover, for  any $\xi\in \mathfrak{t}$ we have
\begin{equation}\label{eq-g-xi}
\mathfrak{g}=\mathfrak{so}^+(1,n+3,\C)=\sum_{j=-r(\xi)}^{r(\xi)}\mathfrak{g}_j^{\xi},\ ~~ \hbox{ with }~~ \mathfrak{g}_j^{\xi}:=\left\{X\in \mathfrak{g}\ |\ [\xi,X]= \sqrt{-1} j X\right\},
\end{equation}
and  $r(\xi)$ being the maximal $j$ satisfying $\mathfrak{g}_j^{\xi}\neq\{0\}$.
For $A\in \mathfrak{so}(1,n+3,\C)$, assume that $A=(\mathbf{c}_{ij})$ with $\mathbf{c}_{ij}$ being $2\times2$ sub-matrix.
 Set
\[\begin{split}&\mathbf{E}_{1j}=\mathbf{E}_{j1}=(\mathbf{c}_{kl}) ,\ \ \hbox{with }\mathbf{c}_{1j}=-I_{1,1}\mathbf{c}_{j1}^t =\left(
                                                                        \begin{array}{cc}
                                                                          1 & i \\
                                                                          1 & i \\
                                                                        \end{array}
                                                                      \right)
\hbox{ and all other $\mathbf{c}_{kl}$ vanishing};\\
&\mathbf{F}_{1j}=\mathbf{F}_{j1}=(\mathbf{c}_{kl}) ,\ \ \hbox{with }\mathbf{c}_{1j}=-I_{1,1}\mathbf{c}_{j1}^t =\left(
                                                                        \begin{array}{cc}
                                                                          1 & i \\
                                                                          -1 & -i \\
                                                                        \end{array}
                                                                      \right),\
\hbox{ and all other $\mathbf{c}_{kl}$ vanishing};\\
& \mathbf{H}_{1j}=\mathbf{H}_{j1}=(\mathbf{c}_{kl}) ,\ \ \hbox{with }\mathbf{c}_{1j}=-I_{1,1}\mathbf{c}_{j1}^t =\left(
                                                                        \begin{array}{cc}
                                                                          1 & -i \\
                                                                          1 & -i \\
                                                                        \end{array}
                                                                      \right)\
\hbox{ and all other $\mathbf{c}_{kl}$ vanishing};\\
 &\mathbf{L}_{1j}=\mathbf{L}_{j1}=(\mathbf{c}_{kl}) ,\ \ \hbox{with }\mathbf{c}_{1j}=-I_{1,1}\mathbf{c}_{j1}^t =\left(
                                                                        \begin{array}{cc}
                                                                          1 & -i \\
                                                                          -1 & i \\
                                                                        \end{array}
                                                                      \right)
\hbox{ and all other $\mathbf{c}_{kl}$ vanishing};\
\end{split}\]
                                                                      here $ 2\leq j \leq  m$.
                                                                       Set
\[\begin{split}&\mathbf{E}_{rj}=\mathbf{E}_{jr}=(\mathbf{c}_{kl}) ,\ \ \hbox{with }\mathbf{c}_{rj}=- \mathbf{c}_{jr}^t =\left(
                                                                        \begin{array}{cc}
                                                                          1 & i \\
                                                                          i & -1 \\
                                                                        \end{array}
                                                                      \right)
\hbox{ and all other $\mathbf{c}_{kl}$ vanishing};\\
&\mathbf{F}_{rj}=\mathbf{F}_{jr}=(\mathbf{c}_{kl}) ,\ \ \hbox{with }\mathbf{c}_{rj}=-\mathbf{c}_{jr}^t= \left(
                                                                        \begin{array}{cc}
                                                                          1 & i \\
                                                                          -i & 1 \\
                                                                        \end{array}
                                                                      \right) \
\hbox{ and all other $\mathbf{c}_{kl}$ vanishing};\\
&\mathbf{H}_{rj}=\mathbf{H}_{jr}=(\mathbf{c}_{kl}) ,\ \ \hbox{with }\mathbf{c}_{rj}= -\mathbf{c}_{jr}^t =\left(
                                                                        \begin{array}{cc}
                                                                          1 & -i \\
                                                                          i & 1 \\
                                                                        \end{array}
                                                                      \right)
\hbox{ and all other $\mathbf{c}_{kl}$ vanishing};\\
&\mathbf{L}_{rj}=\mathbf{L}_{jr}=(\mathbf{c}_{kl}) ,\ \ \hbox{with }\mathbf{c}_{rj}= -\mathbf{c}_{jr}^t =\left(
                                                                        \begin{array}{cc}
                                                                          1 & -i \\
                                                                          -i & -1 \\
                                                                        \end{array}
                                                                      \right)\
 \hbox{ and all other $\mathbf{c}_{kl}$ vanishing};
                                                                      \end{split}\]
                                                                      here $ 2\leq r< j \leq  m$. So it is straightforward to verify
\begin{lemma}\label{lemma-eij}
\begin{equation}\label{eq-hat-xi-gj}
\left\{\begin{split}
\mathfrak{g}_{1}^{\hat\xi_r}&=\hbox{Span}_{\C}\left\{\left.\mathbf{E}_{rj},\mathbf{H}_{rj}\right|1\leq j \leq m, j\neq r\right\},\\ \mathfrak{g}_{-1}^{\hat\xi_r} &=\hbox{Span}_{\C}\left\{\left.\mathbf{F}_{rj},\mathbf{L}_{rj}\right|1\leq j \leq m, j\neq r\right\},\\
  \mathfrak{g}_j^{\hat\xi_r} &=\{0\}, \hbox{ for all }|j|>1.
\end{split}\right.
\end{equation}
So $r(\hat{\xi}_r)=1$ for all $r$, $1\leq r\leq  m$.
\end{lemma}
\vspace{5mm}

Next let us turn to  the simple roots in terms  of $\{\hat{\xi}_j\}$.
Let $\{\tilde\xi_1,\cdots,\tilde\xi_m\}$ be an arbitrary permutation of $\{\hat{\xi}_1,\cdots,\hat{\xi}_m\}$. Let $\tilde{\theta}_j$ be the dual of $\tilde{\xi}_j$, i.e.,
\[\tilde{\theta}_j(\tilde{\xi}_k)=\sqrt{-1}\delta_{jk}, \ \hbox{ for all }\
  j,k=1,\cdots,m.\]  By standard Lie group theory, the roots of  $\mathfrak{so}^+(1,2m-1,\mathbb{C})$ are
\[\{\pm(\tilde{\theta}_j-\tilde{\theta}_k), \ \pm(\tilde{\theta}_j+\tilde{\theta}_k), \ 1\leq j<k\leq m \}.\]
Choose
\[\{ (\tilde{\theta}_j-\tilde{\theta}_k), \  (\tilde{\theta}_j+\tilde{\theta}_k), \ 1\leq j<k\leq m \}\]
to be the set of positive roots. Then it is straightforward to obtain that the simple roots $\{\theta_1,\cdots,\theta_m\}$ of  $\mathfrak{so}^+(1,2m-1,\mathbb{C})$ can be expressed as
$$\theta_j=\tilde{\theta}_j-\tilde{\theta}_{j+1}, \ ~\hbox{ for } j=1,\cdots,m-1,\ ~ \theta_m=\tilde{\theta}_{m-1}+\tilde{\theta}_{m}.$$
Let $\{\xi_1,\cdots,\xi_m\}$ be the dual of the simple roots $\{\theta_1,\cdots,\theta_m\}$, i.e., $\{\xi_j\}$ satisfies
\[\theta_j(\xi_k)= \sqrt{-1}\delta_{jk}\ ~~
\hbox{ for all }\ ~~j,k=1,
\cdots,m.\]
Then we obtain
 \begin{equation}
\xi_j=\sum_{k=1}^{j}\tilde{\xi}_l,\ 1\leq j \leq m-2, \ ~~ \xi_{m-1}=\frac{1}{2}\left(\sum_{j=1}^{m-1}\tilde\xi_j -\tilde\xi_{m}\right),\ ~~ \xi_{m}=\frac{1}{2}\left(\sum_{j=1}^{m-1}\tilde\xi_j +\tilde\xi_{m}\right).
\end{equation}

\begin{lemma}\label{lemma-root}Let $\xi=\xi_{j_1}+\cdots+\xi_{j_r}$ be a canonical element.
Then there exists some $A\in SO^+(1,2m-1, \mathbb{C})$ such that $$Fix_{\xi}=\left\{g\ |\ g\in SO^+(1,2m-1, \mathbb{C}),\ \exp(\pi\xi)g\exp(\pi\xi)^{-1}=g \right\}= A \cdot K^{\mathbb{C}}\cdot A^{-1}$$ if and only if $\xi$ has the form
\[\xi=n_1\tilde\xi_1+n_2\tilde\xi_2+\cdots+n_m\tilde\xi_m\]
with $ n_1,\cdots,n_m$ satisfying the following conditions:
\begin{enumerate}
\item  $n_1$, $\cdots$, $n_m$ $\in\mathbb{Z}$;
\item  $\hbox{max}\{m-1,4\}\geq n_1\geq\cdots\geq n_m=0$;
\item $1\geq n_j-n_{j+1}\geq0$, $j=1,\cdots,m-1$;
\item  $\sharp\{n_j|n_j \hbox{ is odd }\}=2$, or $\sharp\{n_j|n_j  \hbox{ is even }\}=2$.
\end{enumerate}
\end{lemma}
\begin{proof}
Recall that in our case, $K^{\C}=(SO^+(1,3 )\times SO(n))^{\C}$ is defined as the fixed point set of $D=\hbox{diag}\{-I_4,I_{2m-4}\}$.
It is straightforward to see that if $n_1,\cdots,n_m$ satisfy $(1)-(4)$, then there exists some permutation matrix $A\in SO^+(1,2m-1, \mathbb{C})$ such that $\exp(\pi\xi)=ADA^{-1}$ or $\exp(\pi\xi)=-ADA^{-1}$   and then
$Fix_{\xi}= A \cdot K^{\mathbb{C}}\cdot A^{-1}.$

On the other hand, let $\xi$ be a canonical element such that $Fix_{\xi}= A \cdot K^{\mathbb{C}}\cdot A^{-1}$ for some $A\in SO^+(1,2m-1, \mathbb{C})$. Then  $\exp(\pi\xi)=ADA^{-1}$ or $\exp(\pi\xi)=-ADA^{-1}$. Therefore $\exp(\pi\xi)$ has the same eigenvalues as $D$ or $-D$, i.e.,  $\sharp\{n_j|n_j \hbox{ is odd }\}=2$, or $\sharp\{n_j|n_j  \hbox{ is even }\}=2$. And it has only the eigenvalues $\pm1$, showing that it is the sum of some of the elements $\xi_1,\cdots,\xi_{m-2},\xi_{m-1}+\xi_m$. Then (1), (2) and (3) follow easily.
\end{proof}
Applying Lemma \ref{lemma-root}, we obtain immediately
\begin{lemma}\label{lemma-cano}
There are $(m-1)^2$ types of nilpotent Lie subalgebras related to $SO^+(1,n+3)/SO^+(1,3)\times SO(n)$ in $SO^+(1,2m-1, \mathbb{C})$, with $ n+4=2m$.  Up to conjugation, the corresponding canonical elements $\xi$ is given by one of the following  $ (m-1)^2$ elements:
{  \begin{subequations}  \label{eq-root1:1}
\begin{align}
 &\hat\xi_1+\hat\xi_2, \hbox{ with } r(\xi)=2;         \label{eq-root1:1A} \\
  & \hat\xi_1+\hat\xi_2+2\sum_{j=3}^{l}\hat\xi_j,~~ \hbox{ with } r(\xi)\leq 4,~~ ~3\leq l \leq m-1;       \label{eq-root1:1B} \\
  &  3\hat\xi_1+\hat\xi_2+2\sum_{j=3}^{l}\hat\xi_j,~~~~\hat\xi_1+3\hat\xi_2+2\sum_{j=3}^{l}\hat\xi_j,~~ \hbox{ with } r(\xi)=5,~~~ 3\leq l \leq m-1;  \label{eq-root1:1C}\\
& 3\hat\xi_1+\hat\xi_2+4\sum_{j=3}^{l}\hat\xi_j+2\sum_{j=l+1}^{t}\hat\xi_j,~~~~ \hbox{ with } r(\xi)\leq 8,~~~ 3\leq l <t\leq m-1, ;\label{eq-root1:1D}\\
& \hat\xi_1+3\hat\xi_2+4\sum_{j=3}^{l}\hat\xi_j+2\sum_{j=l+1}^{t}\hat\xi_j,~~~~ \hbox{ with } r(\xi)\leq 8,~~~3\leq l <t\leq m-1;\label{eq-root1:1D1}\\
&  \sum_{j=3}^{m}\hat\xi_j, ~~ \hbox{ with } r(\xi)=2;\label{eq-root1:1E}\\
 &  2\hat\xi_1+\sum_{j=3}^{m}\hat\xi_j,~~~~ 2\hat\xi_2+\sum_{j=3}^{m}\hat\xi_j,~~\hbox{ with } r(\xi)=3; \label{eq-root1:1F}\\
 &  2\hat\xi_1+3\sum_{j=3}^{k}\hat\xi_j+ \sum_{j=k+1}^{m}\hat\xi_j,~2\hat\xi_2+3\sum_{j=3}^{k}\hat\xi_j+ \sum_{j=k+1}^{m}\hat\xi_j,~~~~ \hbox{ with } r(\xi)\leq 6,~~~ 3\leq k  \leq m-1; \label{eq-root1:1G}
\end{align}
\end{subequations}
}
\end{lemma}
\begin{proof}
From the proof of Lemma \ref{lemma-root}, we see that to obtain $K^{\mathbb{C}}$ as fixed point set of $Ad\exp(\pi \xi)$, we need that
either $\exp{\pi \xi}=D$ or $\exp{\pi \xi}=-D.$

(i). For the case $\exp{\pi \xi}=D$, we can decide the form of $ \xi$ by max$\{n_j\}$ by using Lemma \ref{lemma-root}.
If $1\leq\hbox{max}\{n_j\}\leq 2,$ there are $m-2$ types of choices of $ \xi$ (up to conjugation):
\[\hat\xi_1+\hat\xi_2,\ ~~ \hat\xi_1+\hat\xi_2+2\sum_{k=1}^j\hat\xi_k,~~~~ 3\leq j\leq m-1. \]
If $\hbox{max}\{n_j\}= 3$, there are $2(m-3)$ types of choices of $ \xi$:
\[ 3\hat\xi_1+\hat\xi_2+2\sum_{k=1}^j\hat\xi_k,~~\hbox{ or }\ \hat\xi_1+3\hat\xi_2+2\sum_{k=1}^j\hat\xi_k,~~~~ 3\leq j\leq m-1. \]
If $\hbox{max}\{n_j\}= 4$, let $t$ be the number of non-zero $n_j$ in $\xi$. We have that $4\leq t\leq m-1$. For every $t$,
there are $2(t-3)$ types of choices of $\xi$:
\[ 3\hat\xi_1+\hat\xi_2+4\sum_{j=3}^{l}\hat\xi_j+2\sum_{j=l+1}^{t}\hat\xi_j,\ \hbox{ or } \hat\xi_1+3\hat\xi_2+4\sum_{j=3}^{l}\hat\xi_j+2\sum_{j=l+1}^{t}\hat\xi_j,\ \ 3\leq l <t\leq m-1.\]
So altogether there are $ (m-3)(m-4)$ types of canonical elements when  $\hbox{max}\{n_j\}= 4$.

(ii). For the case $\exp{\pi\xi}=-D$, we see that there are $2m-3$ types of choices of $\xi$:
  \[\sum_{j=3}^{m}\hat\xi_j,~ 2\hat\xi_1+\sum_{j=3}^{m}\hat\xi_j ,~ 2\hat\xi_2+\sum_{j=3}^{m}\hat\xi_j ,\]\[ 2\hat\xi_1+3\sum_{j=3}^{k}\hat\xi_j+\sum_{j=k+1}^{m}\hat\xi_j\ \hbox{ or } \  2\hat\xi_2+3\sum_{j=3}^{k}\hat\xi_j+\sum_{j=k+1}^{m}\hat\xi_j,\ 3 \leq k \leq m-1.\]

  In a sum, we have that the number of these canonical elements  mentioned above is
$m-2+2(m-3)+(m-3)(m-4)+2m-3=(m-1)^2.$
\end{proof}

From the definition of $\hat\xi_j$ in \eqref{eq-xi}, for every $\xi$ of the form in \eqref{eq-root1:1} it is straightforward to see that \[\mathfrak{g}_{2j+1}^{\xi}\subset \mathfrak{p}^{\mathbb{C}},\ \mathfrak{g}_{2j}^{\xi}\subset \mathfrak{k}^{\mathbb{C}}, \hbox{ for all }\ j\in \mathbb{Z}. \]

Now let's turn to the discussion of the nilpotent Lie algebra  $\sum_{j>0}\mathfrak{g}_{j}^{\xi}$  for some canonical element $\xi$ in Lemma \ref{lemma-cano}.
Moreover, as stated in Theorem 4.13 of \cite{DoWa2}, to classify the normalized potentials, we need only an explicit description of \[\sum_{j\geq0}\mathfrak{g}_{2j+1}^{\xi}=\sum_{j>0}\mathfrak{g}_{j}^{\xi}\cap\mathfrak{p}^{\mathbb{C}}.\]
Set
 \begin{equation}\label{eq-nil}\begin{split}&\mathfrak{N}_{\mathrm{a}}=\hbox{Span}_{\C}\left\{\left.\mathbf{E}_{rj},\mathbf{H}_{rj},  \right|r=1,2,\ 3\leq  j\leq m \right\},\\
&\mathfrak{N}_{\mathrm{b}l}=\hbox{Span}_{\C}\left\{ \mathbf{E}_{r3},\cdots, \mathbf{E}_{rm},\mathbf{H}_{r3},\cdots,\mathbf{H}_{rl}, \mathbf{F}_{r,l+1},\cdots,\mathbf{F}_{rm} |r=1,2 \right\},\\
&\mathfrak{N}_{\mathrm{c}l}=\hbox{Span}_{\C}\left\{ \mathbf{E}_{r3},\cdots, \mathbf{E}_{rm},\mathbf{H}_{13},\cdots,\mathbf{H}_{1m}, \mathbf{H}_{23},\cdots,\mathbf{H}_{2l}, \mathbf{F}_{2,l+1},\cdots,\mathbf{F}_{2m}  |r=1,2 \right\},\\
&\mathfrak{N}_{\mathrm{c}l}'=\hbox{Span}_{\C}\left\{ \mathbf{E}_{r3},\cdots, \mathbf{E}_{rm},\mathbf{H}_{23},\cdots,\mathbf{H}_{2m}, \mathbf{H}_{13},\cdots,\mathbf{H}_{1l}, \mathbf{F}_{1,l+1},\cdots,\mathbf{F}_{1m} |r=1,2  \right\},\\
& \begin{split}
\mathfrak{N}_{\mathrm{d},lt}=\hbox{Span}_{\C} \{& \mathbf{E}_{r3},\cdots, \mathbf{E}_{rm},\mathbf{H}_{1,l+1},\cdots,\mathbf{H}_{1m}, \mathbf{F}_{13},\cdots,\mathbf{F}_{1l}, \mathbf{F}_{23},\cdots,\mathbf{F}_{2m} |r=1,2  \},\end{split}\\
&\begin{split}\mathfrak{N}_{\mathrm{e},lt}=\hbox{Span}_{\C} \{& \mathbf{E}_{r3},\cdots, \mathbf{E}_{rm},\mathbf{H}_{2,l+1},\cdots,\mathbf{H}_{2m}, \mathbf{F}_{23},\cdots,\mathbf{F}_{2l}, \mathbf{F}_{13},\cdots,\mathbf{F}_{1m} |r=1,2  \},\end{split}
\\
&\mathfrak{N}_{\mathrm{f}}=\hbox{Span}_{\C}\left\{\left.\mathbf{E}_{rj},\mathbf{F}_{rj},  \right|r=1,2,\ 3\leq  j\leq m \right\},\\
&\mathfrak{N}_{\mathrm{g}l}=\hbox{Span}_{\C}\left\{ \mathbf{E}_{r3},\cdots, \mathbf{E}_{rm},\mathbf{H}_{13},\cdots,\mathbf{H}_{1m}, \mathbf{F}_{23},\cdots,\mathbf{F}_{2m} |r=1,2 \right\},\\
&\mathfrak{N}_{\mathrm{g}l}'=\hbox{Span}_{\C}\left\{ \mathbf{E}_{r3},\cdots, \mathbf{E}_{rm},\mathbf{H}_{23},\cdots,\mathbf{H}_{2m}, \mathbf{F}_{13},\cdots,\mathbf{F}_{1m} |r=1,2 \right\},\\
 &\mathfrak{N}_{\mathrm{h}l}=\hbox{Span}_{\C}\left\{ \mathbf{E}_{r3},\cdots, \mathbf{E}_{rm},\mathbf{H}_{1,l+1},\cdots,\mathbf{H}_{1m}, \mathbf{F}_{13},\cdots,\mathbf{F}_{1l}, \mathbf{F}_{23},\cdots,\mathbf{F}_{2m} |r=1,2  \right\},\\
 &\mathfrak{N}_{\mathrm{h}l}'=\hbox{Span}_{\C}\left\{ \mathbf{E}_{r3},\cdots, \mathbf{E}_{rm},\mathbf{H}_{2,l+1},\cdots,\mathbf{H}_{2m}, \mathbf{F}_{23},\cdots,\mathbf{F}_{2l}, \mathbf{F}_{13},\cdots,\mathbf{F}_{1m} |r=1,2  \right\}.\\
\end{split}\end{equation}
 Applying Lemma \ref{lemma-eij}, it is straightforward to see that these subsets provide the corresponding $\sum_{j\geq0}\mathfrak{g}_{2j+1}^{\xi}$ related to the corresponding $\xi$ in \eqref{eq-root1:1A}-\eqref{eq-root1:1G} in Lemma \ref{lemma-root} respectively.

\subsection{Proof of the main theorem}

Now let us turn to the proof of Theorem \ref{thm-potential}, i.e., the classification of potentials.  To take a look at the form of potentials, if we consider the elements in \eqref{eq-nil} spanned by $\mathbf{E}_{rj}$, $\mathbf{F}_{rj}$
and $\mathbf{H}_{rj}$ for some fixed $j$, there are four kinds of possible combinations of them:
\[\{\mathbf{E}_{1j},\ \mathbf{E}_{2j},\ \mathbf{H}_{1j},\ \mathbf{H}_{2j}\},\ \{\mathbf{E}_{1j},\ \mathbf{E}_{2j},\ \mathbf{H}_{1j},\ \mathbf{F}_{2j}\},\ \{\mathbf{E}_{1j},\ \mathbf{E}_{2j},\ \mathbf{F}_{1j},\ \mathbf{H}_{2j}\},\ \{\mathbf{E}_{1j},\ \mathbf{E}_{2j},\ \mathbf{F}_{1j},\ \mathbf{F}_{2j}\}.\]
For such an element $X$, assume that
\[ X=\left(
    \begin{array}{cc}
      0 & B_{1} \\
      -B_{1}^tI_{1,3} & 0 \\
    \end{array}
  \right),\ \hbox{ with }\  B_1=(\mathrm{v}_3,\hat{\mathrm{v}}_3,\cdots,\mathrm{v}_{m},\hat{\mathrm{v}}_{m}).\]
  Then $\left( \mathrm{v}_j, \hat{\mathrm{v}}_j\right)$ belongs to one of the following forms, corresponding to above four kinds of possibilities:
  \[ \left(
                                                             \begin{array}{cc}
                                                               h_{1j} &\hat{h}_{1j}\\
                                                               h_{1j} &\hat{h}_{1j}  \\
                                                               h_{3j}  &\hat{h}_{3j}  \\
                                                               ih_{3j}  &i\hat{h}_{3j}  \\
                                                             \end{array}
                                                           \right),\hspace{5mm} \left(
                                                             \begin{array}{cc}
                                                               h_{1j} &\hat{h}_{1j}\\
                                                               h_{1j} &\hat{h}_{1j}  \\
                                                               h_{3j}  & ih_{3j}  \\
                                                               h_{4j}  & ih_{4j}  \\
                                                             \end{array}
                                                           \right) ,\hspace{5mm} \left(
                                                             \begin{array}{cc}
                                                               h_{1j} & ih_{1j} \\
                                                               h_{2j} & ih_{2j}  \\
                                                               h_{3j}  &\hat{h}_{3j}  \\
                                                               ih_{3j}  &i\hat{h}_{3j}  \\
                                                             \end{array}
                                                           \right),\hspace{5mm}  \left(
                                                             \begin{array}{cc}
                                                               h_{1j} & ih_{1j} \\
                                                               h_{2j} & ih_{2j}  \\
                                                               h_{3j}  & ih_{3j}  \\
                                                               h_{4j}  & ih_{4j}  \\
                                                             \end{array}
                                                           \right) .\]
Recalling the condition of $B_1^tI_{1,3}B_1=0$, we obtain the restriction that $h_{3j}^2+h_{4j}^2=0$ for the second case and the restriction that $h_{1j}^2-h_{2j}^2=0$ for the third case. The holomorphic properties show that $h_{4j}=i h_{3j}$ globally or $h_{4j}=-i h_{3j}$ globally for the second case, and $h_{2j}=h_{1j}$ globally or $h_{2j}=- h_{1j}$ globally for the third case.

Applying the condition of $B_1^tI_{1,3}B_1=0$ repeatedly and by a conjugation of $\hbox{diag}\{-I_{1,3},I_{2m-4}\}$ or  $\hbox{diag}\{-I_{1,1},I_2,I_{2m-4}\}$ if necessary, we see that $\left( \mathrm{v}_j, \hat{\mathrm{v}}_j\right)$ has the form
  \[ \left(
                                                             \begin{array}{cc}
                                                               h_{1j} &\hat{h}_{1j}\\
                                                               h_{1j} &\hat{h}_{1j}  \\
                                                               h_{3j}  &\hat{h}_{3j}  \\
                                                               ih_{3j}  &i\hat{h}_{3j}  \\
                                                             \end{array}
                                                           \right)~~ \hbox{ or }~~ \left(
                                                             \begin{array}{cc}
                                                               h_{1j} & ih_{1j} \\
                                                               h_{2j} & ih_{2j}  \\
                                                               h_{3j}  & ih_{3j}  \\
                                                               h_{4j}  & ih_{4j}  \\
                                                             \end{array}
                                                           \right).\]
Theorem \ref{thm-potential} then follows.
 \\

{\large \bf Acknowledgements}\ \ The author is thankful to Prof. Josef Dorfmeister, Prof. Changping Wang and Prof. Xiang Ma for their suggestions and encouragement.  This work is supported by the Project 11571255 of NSFC and the Fundamental Research Funds for the Central Universities.

{\footnotesize
\def\refname{Reference}
}
\vspace{3mm}

{\small\

Peng Wang

Department of Mathematics, Tongji University,

 Siping Road 1239, Shanghai, 200092, P. R. China

{\em E-mail address}: {netwangpeng@tongji.edu.cn}
 }

\end{document}